\documentclass[12pt]{amsart}

\usepackage{amsfonts,amsmath,latexsym,amssymb,verbatim,amsbsy,amsthm}

\usepackage{graphicx}
\usepackage{nicefrac}
\usepackage{dsfont}
\usepackage{mathrsfs}
\usepackage{cancel}
\usepackage[normalem]{ulem}
\usepackage{color}

\usepackage[top=1in, bottom=1in, left=1in, right=1in]{geometry}

\usepackage[dvipsnames]{xcolor}
\usepackage{relsize}

\usepackage[colorlinks=true, pdfstartview=FitV, linkcolor=RoyalBlue,citecolor=ForestGreen, urlcolor=blue]{hyperref}

\numberwithin{equation}{section}
\numberwithin{figure}{section}

\renewcommand{\geq}{\geqslant}

\renewcommand{\leq}{\leqslant}

\newcommand{\be}{\begin{equation}}
\newcommand{\ee}{\end{equation}}

\theoremstyle{plain}
\newtheorem{THEOREM}{Theorem}[section]

\newtheorem{theorem}[THEOREM]{Theorem}
\newtheorem{corollary}[THEOREM]{Corollary}

\theoremstyle{definition}

\theoremstyle{remark}
\theoremstyle{question}

\newtheorem{remark}[THEOREM]{Remark}

\def \dim {d} 

\def \n {\nabla}

\def\bu{{\mathbf u}}
\def\bw{{\mathbf w}}
\def\bk{{\mathbf k}}

\def\one{} 
\def\per{2\pi} 
\def\twopip{} 

\def\hk{(\per)^{\dim}}

\def\ck{(\per)^{-\frac{\dim}{2}}} 
\def\dk{(\per)^{\frac{\dim}{2}}}

\def\eK{e_K}

\def\temp{T}

\def\jnd{{\mathcal J}}

\def\amp{\tau}
\def\om{\Omega}

\def \bk {{\bf k}}
\def \bx {{\mathbf x}}
\def \bxp {{\mathbf x}'}
\def \by {\bxp} 
\def \xp {{x}'}
\def \bu {{\mathbf u}}
\def \bv {{\mathbf v}}
\def \bvp {{\mathbf v}'}
\def \bz {{\mathbf z}}

\def \bc {{\mathbf c}}

\def\A{{\mathbf A}} 
\def\E{E} 
\def\Etotal{{\mathscr E}} 
\def\Egraph{{\sf E}}
\def\Vgraph{{\sf V}}
\def\delE{\delta\Etotal}
\def\div{\nabla_\bx\!\cdot\!}
\def\nablaS{\nabla_{\!\!{}_S}}
\def\nablabu{\nabla \bu}  
\def\avrho{\phi*\rho} 
\def\avrhoz{\phi*\rho_0} 
\def\dmrho{\rho\rho'\dx\dxp} 

\newcommand{\R}{\ensuremath{\mathbb{R}}}   
\newcommand{\T}{\ensuremath{\mathbb{T}}}   

\def\d{{\sf{d}}}

\DeclareMathOperator{\supp}{supp} %
\DeclareMathOperator{\diam}{diam} %
\def \hf {\frac{1}{2}}

\def\suppr{\supp\{\rho(\cdot,t)\}}

\def \dx  {\, \d\bx} 
\def\dxp {\, \d\bxp} 
\def\dvp {\, \d\bv'} 
\def \dy  {\dxp} 
\def \dz  {\, \mathrm{d}\bz}

\def \dv  {\, \mathrm{d}\bv}

\def \ddt  {\frac{\d}{\d t}} 

\def \presure {{\mathbb P}}

\def \Lap {{\mathscr L}} 
\def \etavar {\sigma_\varphi}

\begin{document}

\title[Emergent Behavior in Alignment Dynamics]{On the Mathematics of Swarming:\\Emergent Behavior in Alignment Dynamics}

\author{Eitan Tadmor}
\address{Department of Mathematics and Institute for Physical Sciences \& Technology (IPST), University of Maryland, College Park}
\email{tadmor@umd.edu}

\date{\today}



\thanks{\textbf{Acknowledgment.} Research was supported in part by NSF grants DMS16-13911 and ONR grant N00014-1812465.}

\begin{abstract}
We overview recent developments in the study of alignment hydrodynamics, driven by  a general class of symmetric communication kernels. 
A main question of interest is to characterize  the emergent behavior of such systems, which we quantify in terms of the spectral gap of a \emph{weighted} Laplacian  associated with the alignment operator. 
Our spectral analysis of energy fluctuation covers both long-range and short-range   kernels and does not require thermal equilibrium (no closure for the pressure). In particular, in the prototypical case of metric-based short-range kernels, the  spectral gap admits a lower-bound expressed in terms of the discrete Fourier coefficients of the radial kernel, which enables us to  quantify an emerging flocking behavior for \emph{non-vacuous} solutions.
These large-time behavior results apply as long as the solutions remain smooth. It is known that global smooth solutions exist in one and two spatial dimensions, subject to sub-critical initial data. We settle the question for arbitrary dimension, obtaining non-trivial initial threshold conditions which guarantee existence of multiD global smooth solutions.
\end{abstract}

\maketitle
\setcounter{tocdepth}{1}
\tableofcontents

\section{Introduction}
The starting point of our discussion is the  celebrated Cucker-Smale (CS)  model \cite{CS2007a,CS2007b} which describes the dynamics of $N$ particles, viewed as agents,  with (position, velocity) pairs $(\bx_i(t),\bv_i(t)): \R_+ \mapsto (\om,\R^d)$,
\be\label{eq:CS}
\left\{ \ \ \begin{split}
\dot{\bx}_i(t)&=\bv_i(t)\\
\dot{\bv}_i(t)&=\frac{\amp}{N}\sum_{j=1}^N \phi_{ij}(t)(\bv_j(t)-\bv_i(t)), 
\end{split}\right.
\ee
and subject to prescribed initial conditions, $(\bx_i(0),\bv_i(0))=(\bx_{i0},\bv_{i0})\in (\om,\R^d)$. The ambient space of positions $\om\subset \R^d$ will refer to one of two main scenarios --- either $\om=\R^d$ or $\om=\T^d$. 
System \eqref{eq:CS} is a canonical model for \emph{alignment dynamics}
in which pairwise interactions steer towards average heading.
Alignment originated in pioneering works  \cite{Aok1982,Rey1987,VCBCS1995,CS2007a,CS2007b,CKJRF2002}, as a key ingredient in \emph{self-organization} --- a unity from within which leads to the emergence of higher-order, large-scales patterns.
It is found in ecology --- from flocking of birds and schooling of fish  to swarming bacteria and  insects;  in social dynamics of human interactions --- from alignment of pedestrians and emerging consensus of opinions to markets  and marketing; and in sensor-based networks --- from swarming of mobile robots and control of  UAVs  to macro-molecules and metallic rods.

The dynamics \eqref{eq:CS} governs  pairwise interactions, $\phi_{ij}=\phi(\bx_i,\bx_j)$, dictated by  a scalar \emph{communication kernel} $\phi(\cdot,\cdot)$ with amplitude  $\amp>0$.
We assume that $\phi(\cdot,\cdot)\in L^\infty (\om\times\om)$ is a non-negative symmetric  kernel, properly normalized
\begin{equation}\label{eq:phi}
\int_\om\phi(\bx,\bxp)\dxp \equiv 1, \qquad \phi(\bx,\bxp)=\phi(\bxp,\bx)\geq 0.
\end{equation}

The role for the kernel $\phi$ is context-dependent:
its approximate shape is either derived empirically, deduced from higher order principles, learned from the data or postulated based on phenomenological arguments, e.g., \cite{VCBCS1995,CF2003,CKFL2005,Bal2008,Ka2011,VZ2012,Bia2012,Bia2014, CS2007a,BDT2017,MLK2019,BDT2019,LZTM2019,ST2021} and the references therein. The specific structure of $\phi$, however,  is not necessarily known.
Instead, we ask how different \emph{classes} of communication kernels affect the emergent behavior of \eqref{eq:CS}. Here are few examples for different communication protocols.\newline
A major part of current literature is devoted to the  generic class of  \emph{metric-based} kernels,
$\phi(\bx,\bxp)=\varphi(|\bx-\bxp|)$. A prototype example which goes back to Cucker-Smale \cite{CS2007a}\footnote{We use the abbreviation  $\langle r\rangle :=(1+r^2)^{1/2}$.} $\varphi(r)=\langle r\rangle ^{-\theta}, \ \theta>0$. This is motivated by the phenomenological reasoning that the strength of pairwise interactions is inversely dependent on their relative distance, ``\emph{birds of feather flock together}''. Thus, the emphasize is on the behavior of $\varphi(r)$ for $r\gg1$. This can be contrasted with an opposite protocol based on tendency to attract diverse groups so that $\varphi(r)$ is \emph{increasing} over its compact support, e.g., the \emph{heterophilous} dynamics studied in \cite{MT2014}.\newline
An important protocol of metric-based communication which  will not be analyzed here  includes  the class of \emph{singular} kernels, e.g., the Riesz kernel, 
\begin{equation}\label{eq:singular}
\varphi(r)=\frac{1}{r^{\beta}}, \qquad 0<\beta <d+2,
\end{equation}
 where the communication  emphasizes near-by neighbors, $r\ll 1$, over those farther away, e.g., \cite{Pes2015,CCMP2017,ST2017a,ST2018a,DKRT2018,PS2019,MMPZ2019} and the references therein.\newline
 Another example is the class of \emph{topologically-based} kernels, \cite{Bal2008,ST2020b}, where $\phi(\bx,\bxp)=\varphi(\mu(\bx,\bxp))$ is dictated  by the size of the crowd   in an intermediate  domain of communication ${\mathcal C}(\bx,\bxp)$ enclosed between $\bx$ and $\bxp$,
\begin{equation}\label{eq:whatismu}
 \mu(\bx,\bxp):=\frac{1}{N}\#\{k:\, \bx_k\in {\mathcal C}(\bx,\bxp)\}.
\end{equation}
In particular, if the domain of communication ${\mathcal C}$ is shifted to an $R$-ball centered at $\bx$, one ends up with the \emph{non-symmetric} topological kernel \cite{MT2011}
$\displaystyle \phi(\bx,\bxp)={\varphi(|\bx-\bxp|)}/{\mu(B_R(\bx))}$. 
Other important protocols of communication which  will not be analyzed here  are based on various \emph{random-based} protocols found in chemo- and photo-tactic dynamics,  the Elo rating system, voter and related opinion-based models, a random-batch method and consensus-based optimization, to name but a few, \cite{DAWF2002,BeN2005,CFL2009,GWBL2012,JJ2015,PTTM2017,CCTT2018,JLL2020}.\newline
As a prototype example for `higher order principles' we mention the anticipation-based dynamics \cite{MCEB2015,GTLW2017,ST2021}
\[
\dot{\bv}_i(t)=-\frac{1}{N}\sum_{j=1}^N\nabla_i U(|\bx_i^\tau-\bx_j^\tau|), 
\]
in which agents at time $t$ react to the \emph{anticipated} positions, $\bx_j^\tau(t):=\bx_j(t)+\tau\bv_i(t)$ at time $t+\tau$. The dynamics is governed by a radial potential $U$. Expanding in the (assumed small) $\tau>0$ we find
\begin{equation}\label{eq:3Zone}
\dot{\bv}_i=\frac{\tau}{N}\sum_{j=1}^N \Phi_{ij}(\bv_j-\bv_i)
-\frac{1}{N}\sum_{j=1}^N \nabla U(|\bx_j-\bx_i|).
\end{equation}
The first sum on the right represent the alignment term with amplitude $\tau$, driven by the average Hessian, $\Phi_{ij}=\widetilde{D^2 U}(|\bx_i-\bx_j|)$, and is a prototype for a
a still larger class of symmetric \emph{matrix} communication kernels; the scalar case $\Phi_{ij}=\phi_{ij}{\mathbb I}$ recovers the Cucker-Smale model \eqref{eq:CS}.  In addition to alignment, there is the  second sum on the right of \eqref{eq:3Zone} which governs pairwise  \emph{repulsions} and \emph{attractions}, corresponding to $U'(|\bx_i-\bx_j|)<0$ and respectively, $U'(|\bx_i-\bx_j|)>0$.\newline
A general  protocol for rules of engagement, with pairwise interactions driven by   alignment, repulsion and attraction which are dominant in three different zones of proximity, was proposed in the pioneering work \cite{Rey1987}. We shall focus here  on the emergent behavior of alignment dynamics, and refer to \cite{CFTV2010,CDMBC2007,ST2021} and the references therein,  for results  related to more general protocols. To date,  we still lack a mathematical theory which analyzes the emergent behavior of 
the general class of 3Zone models for collective dynamics.

\subsection{Connectivity}
The large time behavior of \eqref{eq:CS} depends on the time-dependent weighted graph with agents at the $N$ vertices ${\Vgraph}(t)=\{ i\, | \, \bx_i(t)\}$ and time-dependent edges 
${\Egraph}(t)=\{ (i,j) \, | \, i\neq j: \, \phi_{ij}(t) >0\}$.
 The energy fluctuations associated with \eqref{eq:CS}, 
 \[
 \delE(t):= \frac{1}{N^2}\sum_{i,j=1}^N |\bv_i(t)-\bv_j(t)|^2,
 \]
 satisfy
\begin{equation}\label{eq:disfluc}
\begin{split}
\ddt \delE(t) &= - \frac{2\amp}{N^2}\sum_{(i,j)\in {\Egraph}(t)}\!\!\!\phi_{ij}(t) |\bv_i(t)-\bv_j(t)|^2  \leq -\lambda_2(t)\frac{2\amp}{N^2} \sum_{i,j=1}^N |\bv_i(t)-\bv_j(t)|^2.
\end{split}
\end{equation}
The first equality follows directly from \eqref{eq:CS} and the assumed symmetry of the adjacency matrix $\Phi(t)=\{\phi_{ij}(t)\}$. The second inequality is a sharp bound in terms of the \emph{spectral gap}, $\lambda_2(t):=\lambda_2(\Delta_{\Phi(t)})$, of the graph Laplacian  associated with $\Phi(t)$. 
Here the graph Laplacian,  $(\Delta _\Phi)_{\alpha\beta}:=(\sum_{\gamma\neq \alpha} \phi_{\alpha\gamma})\delta_{\alpha\beta}-\phi_{\alpha\beta}(1-\delta_{\alpha\beta})$, and its spectral gap coincides with the Fiedler number, \cite{Fie1973,Fie1989},which encodes the connectivity properties of the weighted graph of agents $({\Vgraph}(t), {\Egraph}(t))$, e.g., \cite{CS2007a}.
Indeed, \eqref{eq:disfluc} tells us that  energy fluctuations are depleted as long as  the graph  remains (algebraically) \emph{connected}
\begin{equation}\label{eq:discfluc}
\delE(t) \leq  exp\left\{-2\amp \!\int^t \!\!\lambda_2(s){\d}s\right\}\delE(0).
\end{equation}
Connectivity --- and hence the large time emergence of flocks or swarms, is guaranteed with \emph{long-range} kernels.
In many realistic configurations, however, the communication among `social particles'  takes place in local neighborhoods induced by \emph{short-range} kernels. \newline
Long and short-range communication kernels will be the topic of the next two sections.
 Long-range kernels maintain connectivity which in turn imply decay of fluctuations around an emergent cluster. The large time dynamics with short-range kernels  is  considerably more complicated --- in particular,  algebraically  connected initial configurations, $\lambda_2(t=0)>0$,  may break down  into two or more disconnected clusters at a finite time so that  $\lambda_2(t_c)=0$. That is, the  dynamics with short-range kernels may or may not be stable, which makes it difficult to trace its flocking behavior. Instead, we study here the flocking/swarming behavior with \emph{large crowds}: large crowd dynamics tends to stabilize the large-time behavior. As already noted by Immanuel Kant in 1784 \cite[p. 11]{Kant1784}
 ``\emph{what seems complex and chaotic in the single individual may be seen from the standpoint of the human race as a whole to be a steady and progressive though slow evolution of its original endowment}''.

\subsection{Hydrodynamic description}
The large crowd dynamics of \eqref{eq:CS} is encoded in terms of the empirical distribution $f_N(t,\bx,\bv)= \frac{1}{N}\sum_{i=1}^N \delta_{\bx_i(t)}(\bx)\otimes \delta_{\bv_i(t)}(\bv)$, which is governed by the kinetic equation in state variables   $(t,\bx,\bv)\in \R_+\times\om\times \R^d$, e.g.,
\cite{HT2008,HL2009,CFTV2010},
\begin{equation}\label{eq:Q}
\begin{split}
\partial_t f_N +\bv\cdot\nabla_\bx f_N = -\amp \nabla_\bv\cdot Q_\phi(f_N,f_N), \qquad (t,\bx,\bv)\in \R_+\times\om\times \R^d.
 \end{split}
\end{equation}
It is driven according to the pairwise  communication protocol\footnote{We abbreviate $\bu=\bu(t,\bx), f'=f(t,\bxp,\bvp)$ and likewise $\square=\square(t,\bx), \ \square'=\square(t,\bxp,\bvp)$ etc.} on the right of \eqref{eq:CS}${}_{2}$
\begin{equation}\label{eq:Qphi}
 Q_\phi(f,f):=\iint_{\R^{d}\times \om}\phi(\bx,\bxp)(\bvp-\bv)f(t,\bx,\bv)f(t,\bxp,\bvp)\dvp\dxp.
\end{equation}
 For  $N\gg 1$, the dynamics of  $f_N(t,\bx,\bv)$ is captured by its first two moments which we assume to exist --- the density $\displaystyle \rho(t,\bx):=\lim_{N\rightarrow \infty} \int_{\R^d} f_N(t,\bx,\bv)\dv$, and   the momentum, $\displaystyle  \rho \bu(t,\bx):=\lim_{N\rightarrow \infty} \int_{\R^d} \bv f_N(t,\bx,\bv)\dv$. 
  They admit the \emph{hydrodynamic description} in state variable $(t,\bx)\in(\R_+\times \om)$,
  \begin{subequations}\label{eqs:hydro}
  \begin{equation}\label{eq:hydro}
 \left\{\ \ \begin{split}
 \rho_t + \n_\bx \cdot (\rho \bu) & = 0, \\
(\rho\bu)_t + \nabla_\bx\cdot (\rho\bu \otimes \bu + \presure) &=\amp\A_\phi(\rho,\bu),
\end{split} \right. \qquad (t,\bx)\in \R_+\times \om.
 \end{equation}
Here, the pressure $\presure$  on the left of \eqref{eq:hydro}${}_2$ is a symmetric positive-definite stress  tensor
\begin{equation}\label{eq:Ptensor}
\presure(t,\bx):= \lim_{N\rightarrow \infty}\int_{\R^d} (\bv - \bu)(\bv - \bu)^\top f_N(t,\bx,\bv) \dv,
\end{equation}
and $\A_\phi$ on the right  of \eqref{eq:hydro}${}_2$ is the communication protocol associated with $\phi$, corresponding to \eqref{eq:Qphi},
\begin{equation}\label{eq:alignment}
\A_\phi(\rho,\bu)(t,\bx):=\int_\om \phi(\bx,\by)\big(\bu(t,\bxp)-\bu(t,\bx)\big)\rho(t,\bx)\rho(t,\by)\dy.
\end{equation} 
\end{subequations}
Observe that system \eqref{eqs:hydro} is not a purely hydrodynamic description at the macroscopic scale:  while  the density  and velocity, $\rho=\rho(t,\bx)$ and $\bu=\bu(t,\bx)$, are governed by the macroscopic balances \eqref{eq:hydro},\eqref{eq:alignment}, the pressure $\presure$  in \eqref{eq:Ptensor}, $\presure=\presure(t,\bx)$,  still requires a  \emph{closure} of the $\bv$-dependent second-order moments of $f_N$.  We recall that in the  case of physical particles, one encounters the canonical closure imposed by Maxwellian equilibrium  and expressed in terms of the  density $\rho$, velocity  $\bu$ and temperature $\temp$, \cite{Cer1990,Lev1996,Gol1997/98},
\[
M_{\{\rho,\bu,\temp\}}(t,\bx,\bv)= \frac{\rho}{(2\pi \temp)^{d/2}}\textnormal{exp}\left(-\frac{|\bv-\bu|^2}{2\temp}\right). 
\]
We mention in this context the  special  cases of \emph{mono-kinetic closure}, $\presure\equiv 0$, associated with the vanishing temperature $M_{\{\rho,\bu,\temp\downarrow 0\}}(t,\bx) =\rho\delta(|\bv-\bu|)$, \cite{KV2015,FK2019}, an entropic-based closure with measured data in \cite{Bia2012} and the   isothermal closure studied in \cite{KMT2015} $\presure=\rho{\mathbb I}_{d\times d}$ (corresponding to constant temperature $\temp$).\newline
 The general case of `social particles', however,  is different: there is no universal Maxwellian closure. This issue will be revisited in section \ref{sec:CT}.
 The  question of closure in the  context of our discussion on   hydrodynamic description for alignment \eqref{eqs:hydro},  is therefore left open. Indeed, we highlight the fact that  the decay of energy fluctuations quantified in the next section applies to general mesoscopic  pressure stress tensors \eqref{eq:Ptensor}. 

\subsection{Fluctuations} 
The total energy of the large-crowd dynamics associated with \eqref{eq:CS} is given by the second moment (which is assumed to exist)
\[
  \rho \E(t,\bx):=\lim_{N\rightarrow \infty} \int_{\R^d} \frac{1}{2}|\bv|^2 f_N(t,\bx,\bv)\dv.
\]
It satisfies  the energy equation
\begin{equation}\label{eq:E}
(\rho \E)_t + \nabla_\bx\cdot (\rho \E\bu + \presure\bu +{\mathbf q}) =-\amp\int\phi(\bx,\bxp)\big(2\E(t,\bx)-\bu\cdot \bu'\big)\rho\rho'\dxp.
\end{equation}
The energy flux on the left of \eqref{eq:E}  is computed as the second  moment of \eqref{eq:Q}, 
\begin{align*}
 \lim_{N\rightarrow \infty}\int \frac{|\bv|^2}{2}&\bv f_N(t,\bx,\bv)\dv =  \bu \lim_{N\rightarrow \infty} \int_{\R^d} \frac{|\bv|^2}{2}f_N \dv + \lim_{N\rightarrow \infty}\int_{\R^d}  \frac{|\bv|^2}{2}(\bv-\bu)f_N \dv\cr
&=   \rho \E\bu + \lim_{N\rightarrow \infty}\int_{\R^d}  \Big( \frac{|\bu|^2}{2}+(\bv-\bu)\cdot \bu+\frac{|\bv-\bu|^2}{2}\Big)(\bv-\bu)f_N \dv  \cr
&=   \rho \E \bu  +  \presure\bu + {\mathbf q},
\end{align*}
expressed in terms of  $\presure$ and the heat flux $\displaystyle  {\mathbf q}(t,\bx):= \lim_{N\rightarrow \infty}\frac{1}{2}\int (\bv-\bu)|\bv-\bu|^2f_N(t,\bx,\bv)\dv$.  The energy production on the right of \eqref{eq:E}  is driven by the  \emph{enstrophy} 
\begin{eqnarray*}
  \iiint \phi(\bx,\bxp) \bv\cdot(\bv-\bvp)ff'\dvp \dv \dxp 
= \int\phi(\bx,\bxp)\big(2\E(\bx)-\bu\cdot \bu'\big)\rho(t,\bx)\rho(t,\bxp)\dxp.
\end{eqnarray*}

The energy   can be decomposed as the sum of kinetic and internal energies,
$\rho \E = \rho \eK + \rho e$, corresponding to the first two terms in the decomposition of kinetic velocity $ \frac{1}{2}|\bv|^2\equiv \frac{1}{2}|\bu|^2+\frac{1}{2}|\bv-\bu|^2+(\bv-\bu)\cdot \bu$,
\[
\begin{split} 
\rho \eK(t,\bx) &:=  \frac{1}{2}\rho{|\bu|^2(t,\bx)}, \\
(\rho e)(t,\bx) &:=\lim_{N\rightarrow \infty}\frac{1}{2}\int|\bv-\bu(t,\bx)|^2 f_N(t,\bx,\bv)\dv.
\end{split}
\]
Let $\delE(t)$ denote the total \emph{energy fluctuations} at time\footnote{Here and below we abbreviate $\dmrho:=\rho(t,\bx)\rho(t,\bxp)\dx\dxp$.} $t$,
\[
\delE(t):= \frac{1}{2m_0}\iint \Big(\frac{1}{2}|\bu(t,\bx)-\bu(t,\bxp)|^2+e(t,\bx)+e(t,\bxp)\Big)\dmrho.
\]
The first integrand quantifies   fluctuations of macroscopic velocities, $\bu(t,\cdot)$, while the last two terms quantify  microscopic fluctuations,  $|\bv-\bu(t,\cdot)|^2$. 
Its decay rate is given by
\begin{equation}\label{eq:book}
\ddt \delE(t)=-\amp\iint\phi(\bx,\bxp)\Big(\frac{1}{2}|\bu(t,\bx)-\bu(t,\bxp)|^2+e(t,\bx)+e(t,\bxp)\Big)\dmrho.
\end{equation}
This follows by integration of the energy equation \eqref{eq:E}: on the left we use $\displaystyle 2m_0\int \rho \E \dx = \iint \left[\frac{1}{2}|\bu|^2+\frac{1}{2}|\bu'|^2 + e +e'\right] \dmrho$ and the fact  that $\displaystyle \ddt \int \rho\bu(t,\bx) \dx \equiv 0$; on the right, we use the symmetric part of the enstrophy 
\[
2\E-\bu\cdot \bu'\equiv \frac{1}{2}|\bu-\bu'|^2 + e+e' + \frac{1}{2}\big(|\bu|^2-|\bu'|^2\big) +(e-e').
\]
Equation \eqref{eq:book}  quantifies  the     decay of  energy fluctuations on the left in terms of the total enstrophy on the right and is a key ingredient in studying the emergent behavior of \eqref{eqs:hydro}.
 \subsection{Flocking/Swarming}
A main feature of the large-crowd hydrodynamics \eqref{eqs:hydro} is the emerging  behavior of \emph{flocking} --- a large-time behavior  characterized by the following two features.

\smallskip\noindent
\begin{subequations}\label{eqs:flocking}
(i) Finite diameter --- the continuum crowd supported in ${\mathcal S}(t)=\suppr$ forms a `flock' with a finite diameter 
\begin{equation}\label{eq:finiteD}
D(t):=\mathop{\sup}_{\bx,\bxp\in {\mathcal S}(t)}|\bx-\bxp| \leq D_+ <\infty, \qquad {\mathcal S}(t):=\suppr;
\end{equation}

\noindent
(ii) Alignment --- velocity fluctuations inside the flock vanish for $t \gg 1$,  
\begin{equation}\label{eq:vanishingu}
\iint|\bu(t,\bx)-\bu(t,\bxp)|^2\dmrho \ \stackrel{t\rightarrow \infty}{\longrightarrow} 0.
\end{equation}
\end{subequations}
 Remark that one can combines \eqref{eq:vanishingu} with   global time invariant of the momentum $\displaystyle \overline{\rho\bu}(t):=\int_{\om}\rho\bu(t,\bx)\dx$ and mass $\displaystyle m(t):=\int_{\om}\rho(t,\bx)\dx$, to deduce the emergent behavior towards the average velocity
$\displaystyle \overline{\bu}_0$, 
 \begin{equation}\label{eq:flock}
 \delE(t) = \int \left[\frac{1}{2}|\bu(t,\bx)-\overline{\bu}_0|^2+e(t,\bx)\right]\rho(t,\bx)\dx \stackrel{t\rightarrow \infty}{\longrightarrow} 0, \qquad \overline{\bu}_0:=\frac{(\overline{\rho\bu})_0}{m_0}.
 \end{equation}
  Thus, large time decay of fluctuations  takes place  while the dynamics is asymptotically aligned along straight particle paths $\bx_c(t) \sim \overline{\bu}_0 t$. We mention in passing that the decay of fluctuations  in the presence of more general protocol of interactions ---  repulsion, attraction and external forcing,  leads to  emergent behavior with different and more `interesting' patterns. For example, when alignment is augmented   by pairwise attraction induced by quadratic potential \eqref{eq:3Zone}, it implies flocking  of agents which are spatially concentrated along particle path  of harmonic oscillators, $(\bx_c(t),\bu_c(t)$, depicted in figure \ref{fig:CS+Qpotential}, $\rho\bu(t,\bx)-m_0\bu_c(t)\delta(\bx-\bx_c(t))  \stackrel{t \rightarrow \infty}{\longrightarrow} 0$, \cite[Theorem 2.5]{ST2020a}. 
 A rich gallery of emerging swarming patterns is found in \cite{CDMBC2007, CDMBC2007,CDP2009,CFTV2010,BCLR2013}.
 
\begin{figure}[ht]
	\includegraphics[width=5in]{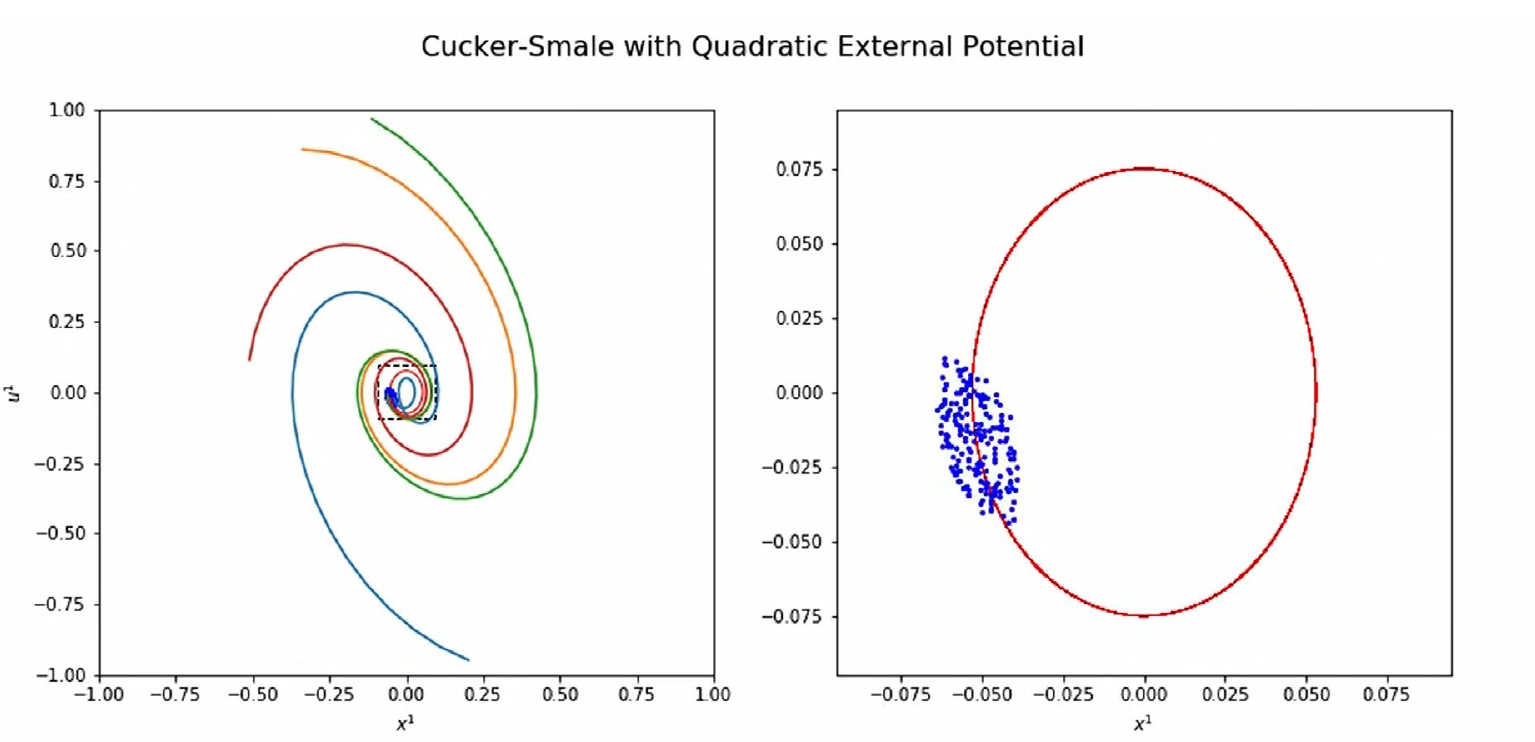}
	\caption{Flocking emerges along harmonic oscillator}\label{fig:CS+Qpotential}
\end{figure}

Flocking is dictated by the enstrophy on the right of \eqref{eq:book}, which in turn is determined by the two types of fluctuations --- the internal energy and kinetic fluctuations. We shall analyze these fluctuations in two separate cases of long-range and short-range communications kernels.

\section{Long-range communication kernels}\label{sec:long}
 Long range kernels maintain global communication  
so that  each  part of the crowd  with mass distribution $\rho(\bx)\dx$ communicates \emph{directly} with every other part with mass distribution $\rho(\bxp)\dxp$.  Global communication is quantified in terms of Pareto-type tail
\begin{equation}\label{eq:D}
\phi(\bx,\bxp) \gtrsim \langle |\bx-\bxp|\rangle^{-\theta}.
\end{equation}
According to \eqref{eq:book}, the decay of energy fluctuations  requires that the diameter of $\suppr$ does not spread too fast.
\begin{corollary}[{\bf Flocking with long-range kernels}]\label{cor:cor1}
Let $(\rho(t,\cdot),\bu(t,\cdot),\presure(t,\cdot))$ be a strong solution of  \eqref{eq:phi},\eqref{eqs:hydro}, driven by long-range kernel \eqref{eq:D}. There holds
\begin{equation}\label{eq:long_range}
\delE(t) \lesssim exp\, \Big\{-\amp \int^t \!\!\langle D(s)\rangle^{-\theta} {\d}s\Big\} \delE(0),  \qquad D(t):=\diam\suppr.
\end{equation}
 \end{corollary}
The flocking behavior of dynamics driven by long-range kernels is determined by two features:
 (i)  the assumed `fat-tail' behavior \eqref{eq:D};  and   (ii) the spread of $\mbox{supp}\{\rho(t,\cdot)\}$. Corollary \eqref{cor:cor1} implies that 
 \begin{equation}\label{eq:Dtbound}
 \text{if} \  D(t)  \lesssim \langle t\rangle^\beta \ \text{then flocking follows for} \ \theta\beta <1.
 \end{equation}
 {\bf `Fat-tailed' kernels}. A prototype example encountered with  uniformly bounded velocity field, in which case $\langle D(t) \rangle \lesssim 2|\bu|_\infty t$ and hence  
  $\beta=1$ implies unconditional flocking for fat-tailed kernels satisfying \eqref{eq:D} with $\theta\in(0,1)$, which in turn implies fluctuations decay rate of order $\lesssim exp\{-t^{1-\theta}\}$. This is the typical scenario for \emph{mono-kinetic closure} $\presure \equiv 0$. In this case, the momentum equation \eqref{eq:hydro}${}_2$ decouples into $d$ scalar transport equations for the components of $\bu$, 
 \begin{equation}\label{eq:transport}
 (\partial_t + \bu\cdot \nabla_\bx) u_i =\int \phi(\bx,\bxp)(u'_i-u_i)\rho(t,\bxp)\dxp, \qquad \bx\in \suppr,
 \end{equation}
 each   satisfies the maximum principle e.g., \cite[Theorem 2.3]{MT2014},\cite[\$1.1]{HT2017}. In fact, in this pressure-less scenario, \eqref{eq:transport} implies  the  uniform decay of  velocity fluctuations is tied to size of $\suppr$
 \begin{align*}
 \ddt V_i(t) & \leq -\amp \langle D(t)\rangle ^{-\theta} V_i(t), \qquad
 V_i(t)=\max_{\bx,\bxp\in {\mathcal S}(t)} |u_i(t,\bx)-u_i(t,\bxp)|\\
 \ddt D(t) &\leq \max_i V_i(t).
 \end{align*}
  It follows that for the   fat-tail alignment \eqref{eq:D} with $\theta<1$, the following functional \cite{HL2009},
  $H(t):=\amp\langle D(t)\rangle^{1-\theta} + (1-\theta)\max_iV_i(t)$, is non-increasing, hence  ${\mathcal S}(t)$ is, in fact,  kept uniformly bounded for all time, 
 \begin{equation}\label{eq:Dleqdplus}
 D(t)\leq D_+,
 \end{equation}
  which in turn implies  exponential flocking $\lesssim exp\{-\amp D_+ t\}$. This flocking result for fat-tailed metric-based kernels $\phi(\bx,\bxp)=\varphi(|\bx-\bxp|)$, with mono-kinetic closure  goes back to Cucker-Smale, e.g., \cite{CS2007a, CS2007b,HT2008,HL2009,CFTV2010,MT2014} and we observe here that it extends to general fat-tailed symmetric kernels. 
  
  Another example for flocking occurs with long-range \emph{matrix-valued} kernels, corresponding to fat-tailed dynamics \eqref{eq:Dtbound} of order $\theta <\nicefrac{2}{3}$: in this case, the diameter of $\suppr$ grows no faster than  $D(t) \lesssim \langle t\rangle^\beta$ with   $\beta<\frac{2}{2-\theta}$, , \cite{ST2021}, leading to flocking decay rate of fractional order $\lesssim exp\{-t^{1-\theta\beta}\}$.

\begin{remark}[{\bf Internal energy}] 
Let $\avrho$ denote the \emph{averaged density} (recall the normalization \eqref{eq:phi})
\begin{equation}\label{eq:philong}
\avrho(t,\bx):=\int \phi(\bx,\bxp)\rho(t,\bxp)\dxp \geq m_0 \phi_-(t), \qquad \phi_-(t)=\min_{\bx,\bxp\in{\mathcal S}(t)} \phi(\bx,\bxp).
\end{equation}
We observe that the contribution of the internal energy to the decay of fluctuations in \eqref{eq:book}  admits the lower bound
\begin{equation}\label{eq:internal}
\begin{split}
\iint  \phi(\bx,\bxp)\big(e+e'\big)\dmrho
 \geq \frac{1}{m_0}\min_\bx\avrho(t,\bx)\iint \big(e+e'\big)\dmrho ,
\end{split} 
\end{equation}
  Hence  the decay of the internal energy  portion  of the fluctuations   is independent of the specifics of the closure relationship: \emph{any} non-negative internal energy is dissipated by fat-tailed kernels, $\phi_-(t)\gtrsim \langle D(t)\rangle^{-\theta}$  such that $\displaystyle \int \langle D(s)\rangle^{-\theta}{\d}s =\infty$.
\end{remark}

\section{Short-range communication kernels}\label{sec:short}

We focus our attention on  the more realistic scenario of short-range communication kernels, and in particular, when $\phi(\bx,\bxp)$ is compactly supported in the vicinity of  diagonal $|\bx-\bxp|< R_0$. This means that alignment takes place in local neighborhoods of size $ < R_0$, which is assumed much smaller than the diameter of the ambient space $\om$. We consider the case of 
 $\per$-periodic torus $\om=\T^{\dim}$.

\subsection{Spectral analysis}
We revisit the two ingredients involved in the decay rate of the energy fluctuations stated in \eqref{eq:book}.

\smallskip\noindent
(i) {\bf Internal energy}. We replace the lower bound \eqref{eq:philong} with 
$\avrho(t,\bx) \geq \rho_-(t)$ (recall the normalization \eqref{eq:phi}). The decay bound  of the internal energy portion in \eqref{eq:book}   for \emph{non-vacuous} flows then reads
\begin{equation}\label{eq:Phishort}
\iint_{\T^\dim\times\T^\dim}  \phi(\bx,\bxp)\big(e+e'\big)\dmrho 
 \geq \frac{\rho_-(t)}{m_0}\iint_{\T^\dim\times\T^\dim} \big(e+e'\big)\dmrho.
 \end{equation}

\smallskip\noindent
(ii) {\bf Kinetic energy}. It remains to bound the contribution of  the kinetic energy fluctuations to the enstrophy on the right of \eqref{eq:book}
\[
\iint_{\T^\dim\times\T^\dim} \phi(\bx,\bxp)|\bu-\bu'|^2\dmrho.
\]
Given the symmetric communication kernel $\phi(\bx,\by)=\phi(\by,\bx)$ we set the \emph{weighted Laplacian} as the Hilbert-Schmidt  operator $\Lap_\rho:L^2(\T^\dim)\rightarrow L^2(\T^\dim)$,
\[
\begin{split}
\Lap_\rho\bw(\bx):=\int_{\T^\dim}\phi(\bx,\by)\Big(\sqrt{\rho(\bxp)}\bw(\bx)-\sqrt{\rho(\bx)}\bw(\by)\Big)\sqrt{\rho(\by)}\dy.
\end{split}
\]
Let $\lambda_k(t)$ be the discrete eigenvalues of $\Lap_{\rho(t)}$ starting with the eigenpair $\lambda_1=0$ (corresponding to eigenfunction $\sqrt{\rho(t,\bx)}\bc$ where $\bc$ is any constant vector in $\R^d$).
The desired lower-bound on the  kinetic energy fluctuations part of the enstrophy
is given by the \emph{spectral gap}, $\lambda_2(\Lap_{\rho(t)})$,
\begin{equation}\label{eq:kinetic}
\begin{split}
\iint_{\T^\dim\times\T^\dim} \!\!\phi(\bx,\bxp)&|\bu-\bu'|^2\dmrho \geq \frac{\lambda_2(\Lap_{\rho(t)})}{m_0} \iint_{\T^\dim\times\T^\dim}  \!\!|\bu-\bu'|^2\dmrho.
\end{split}
\end{equation}
This corresponds to the   bound of  energy fluctuations in the  discrete case \eqref{eq:discfluc}.
The proof is outlined in the end of this section.
\begin{remark}
The  spectral gap bound \eqref{eq:kinetic}  improves the spectral bound derived in \cite[Theorem 1.1]{ST2020b} which required the                                                                                                                                                                                                                                                                                                                                                                                                                      spurious condition $\rho(t,\cdot)\gtrsim \langle t\rangle^{-1/2}$. Instead, \eqref{eq:kinetic} encodes the behavior of $\rho$ through the \emph{weighted} Laplacian $\Lap_{\rho(t)}$; its definition is inspired by the discrete weighted Lapalcian introduced in \cite[\S3]{HT2020}.
\end{remark}

Inserting \eqref{eq:Phishort} and \eqref{eq:kinetic}
into  \eqref{eq:book} yields
the following.
\begin{theorem}[{\bf Flocking with positive spectral gap}]\label{thm:main0}
Let $(\rho(t,\cdot),\bu(t,\cdot))$ be a strong solution of the hydrodynamic system \eqref{eq:phi},\eqref{eqs:hydro} with a mesoscopic pressure $\presure(t,\cdot)$, subject to non-vacuous initial data $(\rho_0>0, \bu_0,\presure_0)$. Then  we the following flocking  decay estimate holds
 \begin{equation}\label{eq:main0}
 \begin{split}
\delE(t) \leq 
 exp\,\Big\{-2\amp\int^t \min\big\{\lambda_2(\Lap_{\rho(s)}),\rho_-(\tau)\big\}{\d}s\Big\} \delE(0).
 \end{split}
 \end{equation}
 \end{theorem}
 \noindent
The energy fluctuations bound \eqref{eq:main0} is in complete analogy with the discrete bound \eqref{eq:discfluc}. In particular, if $\lambda_2(\Lap_{\rho(t)})$ and $\rho_-(t)$ have fat-tailed decay in time, then they `communicate' strong enough  alignment to imply  the flocking behavior sought in \eqref{eq:flock},
\[
\delE(t) = \int_{\T^\dim} \left[\frac{1}{2}|\bu(t,\bx)-\overline{\bu}_0|^2+e(t,\bx)\right]\rho(t,\bx)\dx \rightarrow 0,
\]
and it comes with the additional decay of the internal energy.

\begin{proof}[Proof of the spectral gap bound \eqref{eq:kinetic}]
Given the symmetric communication kernel $\phi(\bx,\by)=\phi(\by,\bx) \in L^\infty(\T^\dim\times\T^\dim)$ we set the \emph{weighted Laplacian operator}, $\Lap_\rho:=\Lambda_\rho-{\mathscr A}_\rho$ where  $\Lambda_\rho:L^2_\rho(\T^\dim)\rightarrow L^2_\rho(\T^\dim)$ and ${\mathscr A}_\rho:L^2_\rho(\T^\dim)\rightarrow L^2_\rho(\T^\dim)$ are a multiplication operator and, respectively, Hilbert-Schmidt operator on $\displaystyle L^2_\rho(\T^\dim):=\Big\{\bw  : \int_{\T^\dim} |\bw|^2\rho \dx<\infty \Big\}$,   which involve the positive weight function\footnote{In agreement with the standard convention of keeping \emph{positive} graph Laplacians} $\rho>0$
\[
\begin{split}
\Lambda_\rho\bw(\bx):=\avrho(\bx) \bw(\bx), \qquad {\mathscr A}_\rho\bw(\bx):=\sqrt{\rho(\bx)}\int_{\T^\dim} \phi(\bx,\by)\bw(\by)\sqrt{\rho(\by)}\dy.
\end{split}
\]
The Laplacian $\Lap_\rho$ is a symmetric non-negative operator in $L^2_\rho(\om)$:
\[
\begin{split}
(\Lap_\rho \sqrt{\rho}\bw,\sqrt{\rho}\bw)&=
\iint_{\T^\dim\times\T^\dim} \phi(\bx,\by)\rho(\by)\rho(\bx)|\bw(\bx)|^2\dx\dy\\
& \quad -
\iint_{\T^\dim\times\T^\dim} \sqrt{\rho(\bx)}\phi(\bx,\by)\sqrt{\rho(\by)}\big\langle \sqrt{\rho(\by)}\bw(\by),\sqrt{\rho(\bx)}\bw(\bx)\big\rangle \dx\dy\\
& = \frac{1}{2}\iint_{\T^\dim\times\T^\dim} \phi(\bx,\by)|\bw(\bx)-\bw(\by)|^2\rho(\bx)\rho(\by)\dx\dy.
\end{split}
\]
Let $(\lambda_k\geq0,\bw_k(\bx))$ be the sequence of discrete eigen-pairs of $\Lap_\rho$ starting with the eigenpair $(\lambda_1=0,\bw_1(\bx)=\sqrt{\rho(\bx)}\bc)$, where $\bc$ is any constant vector in $\R^d$
\[
\begin{split}
\Lap_\rho\bw_1&= \Lambda_\rho(\sqrt{\rho(\bx)}\bc)-{\mathscr A}_\rho(\sqrt{\rho(\bx)}\bc)\\
 & = \sqrt{\rho(\bx)}\int_{\T^\dim}\phi(\bx,\by)
\Big(\rho(\by)-\sqrt{\rho(\by)}\sqrt{\rho(\by)}\Big)\dy\times \bc =0.
\end{split}
\]
We then have
\begin{equation}\label{eq:lam2}
\begin{split}
\lambda_2(\Lap_\rho) &= \inf_{\sqrt{\rho}\bw \perp \sqrt{\rho}\bc} \frac{(\Lap_\rho \sqrt{\rho}\bw,\sqrt{\rho}\bw)}{(\sqrt{\rho}\bw,\sqrt{\rho}\bw)} \\
&= \inf \left\{ \frac{\displaystyle \frac{1}{2}\iint \phi(\bx,\by)|\bw-\bw')|^2\dmrho}{\displaystyle \frac{1}{2m_0}\iint |\bw-\bw'|^2\dmrho} \ \ : \  \ \int \rho \bw =0 \right\}.
\end{split}
\end{equation}
The desired lower-bound of the  kinetic energy fluctuations \eqref{eq:kinetic} follows.
\end{proof}

\subsection{Metric-based communication kernels}
The main difficulty with theorem \ref{thm:main0}  is access to the spectral gap $\lambda_2(\Lap_{\rho(t)})$.
To this end, we  restrict attention to  \emph{radial communication kernel}, $\phi(\bx,\bxp)=\varphi(|\bx-\bxp|)$, over the $\per$-periodic torus $\T^{\dim}$. Short-range communication refers to 'thin tail' kernels, and in particular --- kernels with finite support, (much) smaller than the diameter of the `crowd', $\text{supp}\{\varphi(\cdot)\} < \per$, so that there is lack of global direct communication. Instead, decay of energy fluctuations (and hence flocking) persists for non-vacuous configurations quantified below. We denote 
\[
 c_\rho(t):= \frac{\rho_-(t)}{\rho_+(t)}, \qquad 
\rho_\pm(t):=\substack{{\displaystyle \max}_\bx\\ {\displaystyle \min}_\bx}\rho(t,\bx).
\]
\begin{theorem}[{\bf Flocking with short-range kernels}]\label{thm:main1}
Consider the hydrodynamic system \eqref{eqs:hydro} over the $\per$-periodic torus $\T^{\dim}$, driven by a non-negative radial communication kernel, $\phi(\bx,\bxp)=\varphi(|\bx-\bxp|)$, with unit mass
$\displaystyle \int_{{\mathbb T}^\dim} \varphi(|\bx|)\dx=1$. 
 Let $(\rho,\bu, \presure)$ be a strong solution  subject to  non-vacuous initial data $(\rho_0>0, \bu_0,\presure_0)$.  There exists a constant 
\[
  \etavar:=  1-\max_{\bk \neq \{\mathbf 0\}}\int_{{\mathbb T}^\dim} \varphi(|\bx|)\cos\big(\twopip\bk\cdot\bx\big)\dx >0
\]
such that
\begin{equation}\label{eq:lamtwo}
\lambda_2(\Lap_{\rho(t)})\geq \frac{1}{2}\etavar  c_\rho(t)\rho_-(t), 
\end{equation}
 and the following bound on the decay of energy fluctuations holds
\begin{equation}\label{eq:delE}
\delE(t)   \leq exp\left\{\displaystyle -\amp \etavar \!\int^t \!\!c_\rho(s)\rho_-(\tau){\d}s\right\} \delE(0).
\end{equation}
\end{theorem}
\begin{remark}[{\bf Optimality of the spectral gap bound?}]\label{rem:crhot}
The obvious bound 
\[
\lambda_2(L_{\rho(t)}) \geq m_0\cdot\min \phi(\bx,\bxp),
\]
 implies that when $\phi$ is a long-range communication kernel, namely --- when \eqref{eq:D} holds with $\theta<1$, then it induces an \emph{unconditional flocking}.
The question of flocking for short-range kernels is more subtle:  theorem \ref{thm:main1} shifts the burden of proving flocking in this case to a question of non-vacuous  bounded density, $\rho_-(t) \gtrsim \langle t\rangle^{-1/2}$. 
We raise the question whether an improved bound  of the spectral gap holds --- independent of the aspect ratio $c_\rho$, 
$\lambda_2(\Lap_{\rho(t)}) \gtrsim \rho_-(t)$: this would imply flocking for `fat-tailed' density such that $\displaystyle \rho_-(t) \gtrsim \langle t\rangle^{-1}$.
\end{remark}

\begin{proof}[Proof of theorem \ref{thm:main1}]
We begin with the following  Poincar\'{e} inequality, corresponding to the   bound of discrete energy fluctuations in \eqref{eq:disfluc}: for all $\per$-periodic $\bw\in L^2({\mathbb T}^\dim)$  there holds
\begin{equation}\label{eq:Fbound}
\iint_{\T^\dim\times\T^\dim}\varphi(|\bx-\bxp|)|\bw(\bx)-\bw(\bxp)|^2\dx\dxp
\geq\frac{\etavar}{(2\pi)^d}\iint_{\T^\dim\times\T^\dim} |\bw(\bx)-\bw(\bxp)|^2\dx\dxp.
\end{equation}
Indeed, expressed in terms of  the   Fourier expansion 
\[
\bw(\bx)=\ck\sum_\bk \widehat{\bw}(\bk)e^{i\twopip\bk\cdot\bx}, \qquad \widehat{\bw}(\bk) =\ck \int_{{\mathbb T}^\dim}\bw(\bx)e^{-i\twopip\bk\cdot\bx}\dx,
\]
the integral on the right of \eqref{eq:Fbound} amounts  to
\begin{equation}\label{eq:RHS}
\begin{split}
\frac{1}{2}\iint_{\T^\dim\times\T^\dim} |\bw(\bx)-\bw(\bxp)|^2\dx\dxp &=\hk\!\int_{{\mathbb T}^\dim}|\bw(\bx)|^2\dx-\Big| \int_{{\mathbb T}^\dim}\bw(\bx)\dx\Big|^2 \\
 &= \hk\sum_{\bk \neq {\mathbf 0}}|\widehat{\bw}(\bk)|^2.
 \end{split}
\end{equation}
 Computing the convolution terms on the left of \eqref{eq:Fbound}, 
 $\widehat{\langle \varphi*\bw\rangle }(\bk)= \dk\langle \widehat{\varphi}(\bk),\widehat{\bw}(\bk)\rangle$, and using the assumed unit mass $\widehat{\varphi}(0)=\one\ck$,  yields \eqref{eq:Fbound}:
\begin{equation}\label{eq:LHS}
\begin{split}
\frac{1}{2}\iint_{\T^\dim\times\T^\dim} &\varphi(|\bx-\bxp|)|\bw-\bw'|^2\dx\dxp  \\
& = \int_{{\mathbb T}^\dim} |\bw(\bx)|^2\dx - {\mathrm Re}\int_{{\mathbb T}^\dim}\big\langle \bw(\bx), (\varphi*\bw)(\bx) \big\rangle \dx\\
& =
\sum_{\bk \neq {\mathbf 0}} \left(1- {\mathrm Re}\int_{{\mathbb T}^\dim} \varphi(|\bx|)e^{i\twopip\bk\cdot\bx} \dx\right) |\widehat{\bw}(\bk)|^2 \\
& \geq   \left(1-\max_{\bk \neq {\mathbf 0}}\int_{{\mathbb T}^\dim} \varphi(\bx) \cos(\twopip\bk\cdot\bx)\dx\right) \sum_{\bk \neq{\mathbf 0}}|\widehat{\bw}(\bk)|^2.
\end{split}
\end{equation}
Using \eqref{eq:Fbound} we compute the lower-bound
\[
\begin{split}
\iint_{\T^\dim\times\T^\dim}&\varphi(|\bx-\bxp|)|\bu-\bu'|^2\dmrho 
   \\
  &  \geq \rho_-^2(t) \iint\varphi(|\bx-\bxp|)|\bu-\bu'|^2\dx\dxp 
   \geq \frac{\etavar}{(2\pi)^\dim} \rho_-^2(t)\iint|\bu-\bu'|^2\dx\dxp \\
     & \geq \frac{\etavar}{(2\pi)^\dim} \frac{\rho_-^2(t)}{\rho_+(t)}(2\pi)^\dim\iint|\bu-\bu'|^2\rho(t,\bx)\dx\dxp
    \geq \etavar \frac{\rho_-^2(t)}{\rho_+(t)}\frac{1}{m_0}\iint|\bu-\overline{\bu}_0|^2\dmrho\\
   & = \etavar \frac{\rho_-^2(t)}{\rho_+(t)}\frac{1}{2m_0}\iint_{\T^\dim\times\T^\dim}|\bu-\bu'|^2\dmrho.
 \end{split}
\]
(Recall $\displaystyle \overline{\bu}=\frac{\overline{\rho\bu}}{m_0}$ so the fourth inequality follows from $\displaystyle \int|\bu-{\mathbf c}|^2\rho \geq \int |\bu-\overline{\bu}|^2\rho$ for all constant vectors ${\mathbf c}$).
We now deduce \eqref{eq:lamtwo} from the optimaility of $\lambda_2(\Lap_\rho)$ in \eqref{eq:lam2}; observe that the eigenspace associated with $\lambda_2(\Lap_{\rho(t)})$ remains uniformly bounded away from the eigenspace of constants associated with $\lambda_1(\Lap_{\rho(t)})=0$.
Moreover, 
$\displaystyle \frac{1}{2}\etavar c_\rho\rho_-(t)  \leq \rho_-(t)
 \leq \avrho(t)$
 and the \eqref{eq:delE} follows from theorem \ref{thm:main0}.
\end{proof}

Let $\overline{m}_0$ denote the average mass $\displaystyle \overline{m}_0:=\frac{m_0}{\hk}$. Theorem \ref{thm:main1} implies that  as long density fluctuations remains below a specified  threshold; there exists a constant $c<1$ such that
\begin{equation}\label{eq:denvar}
 \rho_+(t)-\rho_-(t) \leq (1-c) \overline{m}_0, \qquad c<1;
 \end{equation}
then $c_\rho(t) \geq c$ and hence $\displaystyle \rho_-(t) \geq \overline{m}_0c$. We end up with the exponential flocking bound
  \begin{equation}\label{eq:main1}
 \delE(t)   \leq e^{- \delta\etavar  t} \delE(0), \qquad 
 \delta:=\amp\overline{m}_0c^2, \ \ 0<c<1.
 \end{equation}
 This echos a similar result for \emph{first-order} consensus dynamics encoded in terms of $\{\bx_i\} \leadsto \rho$: if the variation of the density remains below a specified $\varphi$-dependent  threshold then smooth solutions of   approach a consensus, \cite{GPY2017}. In both cases, the threshold, quantified  in terms of  the Fourier transform of $\varphi$, dictates flocking/consensus for short-range kernels. We close the section with two examples.
  
\subsection{Examples}\mbox{ }

\smallskip\noindent
{\bf Example 1}. Consider the 1D dynamics over the $2\pi$-torus $\T$ driven by the communication kernel\footnote{$\mathds{1}$ denotes the  characteristic function of the unit ball $\mathds{1}(r):=\left\{\begin{array}{ll} 1 & 0\leq r\leq 1\\
0 & r>1.\end{array}\right.$}
$\phi(x,\xp)=\hf\mathds{1}(|x-\xp|)$. The corresponding  threshold 
 is given by $\displaystyle \etavar=1-\sin(1) \sim 0.158$. It follows that if the density has a finite variation so that \eqref{eq:denvar} holds, then  the 1D CS dynamics \eqref{eqs:hydro} admits exponentially converging flocking $\lesssim e^{-0.158\delta t}$. 
  In fact, in a recent work of Dietert and  Shvydkoy \cite[Theorem 1.3]{DS2020} it was shown  that in the special  1D case,  \emph{any}  discrete CS dynamics \eqref{eq:CS} with non-trivial communication kernel admits flocking  rate  of order $\lesssim \left(\nicefrac{\ln(t)}{t}\right)^{1/5}$.
Here, by restricting attention to   large-crowd dynamics with  slowly varying density, we  improve the flocking result  to exponential rate.

\smallskip\noindent
{\bf Example 2}. Consider the 2D dynamics over the $2\pi$-periodic torus $\T^2$ driven by the communication kernel $\phi(\bx,\bxp)=\frac{1}{\pi}\mathds{1}(|\bx-\bxp|)$. The Fourier coefficients of the radial $\phi$ are given by, \cite{PST1993}, $\displaystyle 2\pi \widehat{\varphi}(\bk)= \frac{2}{|\bk|}J_1(|\bk|)$ and hence 
\[
 \etavar=1-\frac{J_1(1)}{\pi} \sim 0.86.
\]
 It follows that if the density variation remains with the range 
\eqref{eq:denvar} then  the 2D CS dynamics \eqref{eqs:hydro} admits exponentially converging flocking $\lesssim e^{-0.86 \delta t}$.

\section{Global smooth solutions: critical thresholds}
 The hydrodynamics \eqref{eqs:hydro} is driven by two competing mechanisms: a generic effect  in  Eulerian dynamics of \emph{steepening local fluctuations} which may lead to finite-time blow-up when $\displaystyle \lim_{(t,\bx) \uparrow (t_c,\bx_c)}\nabla_\bx\cdot\bu(t,\bx)=-\infty$,  and  alignment which prevents the formation of  shock discontinuities, 
 \begin{equation}\label{eq:C0}
 \nabla_\bx \cdot \bu(t,\cdot) \geq -C_0 >-\infty,
 \end{equation}
 as $|\bu-\bu'|^2\dmrho \stackrel{t\rightarrow \infty}\longrightarrow 0$. The outcome of this competition determines whether \eqref{eqs:hydro} admits  strong solutions sought in theorems \ref{thm:main0} and \ref{thm:main1}. 
 The global existence results  available in current literature  are almost exclusively devoted to mono-kinetic closure $\presure \equiv 0$, 
\begin{equation}\label{eq:noP}
 \left\{\ \ \begin{split}
 \rho_t + \n_\bx \cdot (\rho \bu) & = 0, \\
(\rho\bu)_t + \nabla_\bx\cdot (\rho\bu \otimes \bu ) &=\amp\int_{\R^d} \phi(\bx,\by)\big((\bu(t,\by)-\bu(t,\bx)\big)\rho(t,\by)\rho(t,\bx)\dy.
\end{split} \right. 
 \end{equation}
This is the default closure relation used in much of the `macroscopic' literature, e.g.,  \cite{BDT2017,BDT2019} and the references therein.

 \medskip\noindent
 Although the repulsive forcing of pressure is missing, system \eqref{eq:noP} is still driven by a competition between nonlinear advection and alignment.
Indeed, the alignment hydrodynamics  with or without pressure, \eqref{eqs:hydro} or \eqref{eq:noP},  may form finite time  
shock-discontinuities coupled with the emergence of  Dirac mass which require their  interpretation  as  \emph{weak solutions}, e.g., \cite{KMT2013}. A proper notion of  weak solutions which enforces uniqueness within an admissible class of solutions is still missing. A theory of \emph{dissipative weak solutions}
was developed in \cite{CFGS2017} in which the authors establish a weak-strong 
uniqueness principle.  
 Existence of global strong solutions, on the other hand,  depend on certain \emph{critical thresholds} in the space of initial configurations.  
 
\subsection{Critical thresholds}\label{sec:CT}
To motivate our choice for initial thresholds, we turn to  discuss the \emph{thermodynamics} of the general alignment system \eqref{eqs:hydro}. 
 The entities governed by collective description \eqref{eq:CS}  are fundamentally different then physical particles. While physical particles are driven by forces induced by the
environment of other particles, the  `social particles' we consider here are driven by \emph{probing}
the environment --- living organisms, human interactions and sensor-based agents have senses and sensors
with which they actively probe the environment. In particular, such social agents 
receive energy from the outside, thus forming  thermo-dynamically open systems.  This is particularly apparent in self-organization of biological agents that receive energy from the outside, , thus forming  thermo-dynamically open systems, e.g., \cite{Kar2008}.
Accordingly, the meso-scopic description for flocking cannot not be expected to provide a self-contained  notion of thermodynamic closure sought in \eqref{eqs:hydro}, and as such, there is no universal Maxwellian for thermal equilibrium. As noted in \cite[\S1.1]{VZ2012},
`The source of energy making the motion possible ... are not relevant'.
Nevertheless, lack of thermal equilibrium  in the form of certain closure \emph{equalities} can be substituted with certain \emph{inequalities} which are compatible with the decay of internal energy fluctuations in \eqref{eq:book}.  
To this end, we  trace the separate contributions of the kinetic and internal energies.  Multiply the momentum \eqref{eq:hydro}${}_2$ by $\bu$, we find
\[
\begin{split}
\partial_t (\rho \eK)  + \nabla_\bx\cdot \Big(\rho \bu \frac{|\bu|^2}{2} +\presure\bu\Big)  + \sum_{i,j}P_{ij}\frac{\partial u_i}{\partial x_j}  
 =  - \amp \int \phi(\bx,\bxp) \big(|\bu|^2 - \bu\cdot \bu'\big)\rho\rho'\dxp. 
\end{split}
\]
Subtract the portion of kinetic energy   from the balance of total energy in \eqref{eq:E} we find the dynamics for the internal energy,
expressed in terms of the average density \eqref{eq:philong}, $\avrho=\int \phi(\bx,\bxp)\rho(t,\bxp)\dxp$ and the $d\times d$ velocity gradient matrix $\displaystyle \nablabu:=\Big\{\frac{\partial u_i}{\partial x_j}\Big\}$,
\[
\begin{split}
\partial_t (\rho e) + \div (\bu\rho e+{\mathbf q}) =  -\textnormal{trace}\Big(\presure\nablabu\Big)
  -2\amp \rho e\avrho(t,\bx). 
\end{split}
\]
Since $\displaystyle \textnormal{trace}(\presure)=\int |\bv-\bu|^2f(\bv)\dv= 2\rho e>0$,
we can rewrite the equation governing the internal energy
\begin{equation}\label{eq:re-1}
\begin{split}
\partial_t (\rho e) + \div (\bu\rho e+{\mathbf q}) 
 = -2\jnd\rho e 
\end{split}
\end{equation}
where $\jnd=\jnd(t,\bx)$ is  given in terms of the normalized pressure $\displaystyle \overline{\presure}= \frac{\presure}{\textnormal{trace}\left(\presure\right)}$
\begin{equation}
\jnd(t,\bx):=  \textnormal{trace}(\overline{\presure}\nablabu)(t,\bx)+ \amp \avrho(t,\bx),\qquad \overline{\presure}= \frac{\presure}{\textnormal{trace}\left(\presure\right)}.
\end{equation}

We do \underline{not} enforce  any specific form for  closure  of the internal energy.  Instead, we explore the flocking dynamics  subject to a rather general set thermodynamic configurations  with the minimal assumption, in agreement with \eqref{eq:book}, that the total amount of internal energy is non-increasing.
Integration yields
\[
\int \rho e(t,\bx)\dx \leq \textnormal{exp}\left\{-2\int^t \jnd_-(s)\text{d}s\right\}\int (\rho e)_0(\bx)\dx , \qquad \jnd_-(t):=\min_{\bx} \, \jnd(t,\bx). 
\]
 Thus, a non-increasing total energy is tied to the inequality $\jnd_-(t) \geq 0$ (and in fact, if there is no heat flux, ${\mathbf q}\equiv 0$, then 
 \eqref{eq:re-1} would yield a \emph{uniform} decay  of  internal energy in thsi case).
   A simple exercise shows that since  $\overline{\presure}$ is a symmetric positive definite matrix with trace $1$, then 
 $\textnormal{trace}(\overline{\presure}{\mathbb M}) \geq \lambda_{\textnormal{min}}({\mathbb M}_S)$ for any matrix ${\mathbb M}$ with symmetric part ${\mathbb M}_S:=\frac{1}{2}({\mathbb M}+{\mathbb M}^\top)$. In particular we have the following lower-bound in terms of the symmetric gradient $\nablaS \bu$
 \begin{equation}\label{eq:btrace}
  \textnormal{trace}\Big(\overline{\presure}\nablabu\Big) \geq \lambda_{\textnormal{min}}(\nablaS \bu),\qquad \nablaS \bu:=\frac{1}{2}\left(\partial_j u_i+\partial_i u_j\right).
  \end{equation} 
 We therefore postulate the following critical threshold requirement: there exists $\eta_c\geq 0$ such that
 \begin{equation}\label{eq:CT}
\eta(\rho,\bu)(t,\bx):= \lambda_{\textnormal{min}}(\nablaS \bu)(t,\bx) + \amp \avrho(t,\bx) \geq \eta_c, \qquad \eta_c\geq0.
 \end{equation}
 Observe that the threshold \eqref{eq:CT} is independent of the thermodynamic state of the system and it will guarantee the decay $\displaystyle \int \rho e(t,\bx)\dx \lesssim e^{-2\eta_ct}$. The key question is whether  such threshold persists in time. 
 
 At this point it instructive to compare \eqref{eq:CT} with the known results of  global regularity in dimension $d=1,2$.
A global smooth solution in the 1D case exists  if and only if the initial configuration satisfies the threshold $u'_0(x)+\amp\varphi*\rho_0(x)\geq 0$, \cite{CCTT2016,TT2014,Les2020} and the extension to \emph{uni-directional} flows is found in \cite{LS2019};  this corresponds to \eqref{eq:CT}${}_{t=0}$ with $\eta_c=0$. A sufficient  threshold for 2D regularity was derived in \cite{HT2017}. It requires a lower bound on the initial divergence \emph{and} an upper-bound on the spectral gap (recall \eqref{eq:Dleqdplus})
\[
\begin{split}
\nabla_\bx\cdot \bu_0 + \amp\varphi*\rho_0 &>0,\\
(\lambda_2-\lambda_1)(\nablaS \bu_0) & \leq \delta_0, \qquad \delta_0=:\frac{1}{2}m_0\phi(D_+),
\end{split}
\] 
which imply that \eqref{eq:CT}${}_{t=0}$ holds with 
$\eta_c=\frac{1}{2}(\min_\bx \varphi*\rho_0-\delta_0)$.
 Existence of strong solutions  in $d\geq 3$-dimensions  for  `small data'  can be found in \cite{HKK2014,Sh2019}).
The next result settles the open question of existence of strong solutions in  
$d\geq 3$ dimensions.

 \begin{theorem}[{\bf Existence of global strong  solutions for multiD Euler alignment system} \cite{Tad2021}]\label{thm:main2}
 Consider the Euler alignment system \eqref{eq:noP} subject to non-vacuous initial data, $(\rho_0,\bu_0)\in H^m\times H^{m+1}$  with initial velocity of finite variation\footnote{For the specific value of the constant $C_\phi$ see \cite[(2.4)]{HT2017}} $\max_\bx|\bu'_0-\bu_0| \leq C_\phi$.  If the  initial conditions satisfy the threshold condition  
 \[
 \lambda_{\textnormal{min}}(\nablaS \bu_0)(\bx) + \amp \avrhoz(\bx) \geq \eta_c, \qquad \eta_c=\frac{1}{2}(\rho_0)_->0,
 \]
 then this threshold persists in time,
 \[
 \lambda_{\textnormal{min}}(\nablaS \bu)(t,\bx) + \amp \avrho(t,\bx) \geq \eta_c, \qquad t>0,
 \]
  and \eqref{eq:noP} admits global smooth solution
 $(\rho,\bu)\in C([0,T]; H^m\times H^{m+1})$.
 \end{theorem}
 The main feature here is that the initial threshold, $\eta(\rho_0,\bu_0)\geq \eta_c$, forms an invariant region in $(\rho,\bu)$-configuration space,
  \[
\eta(\rho_0,\bu_0)(\bx) \geq \eta_c \ \leadsto \ \eta(\rho,\bu)(t,\bx)\geq \eta_c,
 \]
 and therefore
 \[
 \nabla_\bx \cdot \bu(t,\bx) = \sum_\lambda \lambda(\nablaS \bu) \geq d\lambda_{\textnormal{min}}(\nablaS \bu)(t,\bx) 
 \geq  d \big(\eta(\rho,\bu)(t,\bx) -\amp \avrho(t,\bx)\big).
 \] 
It follows that \eqref{eq:C0} holds with $\displaystyle C_0=d\big(\amp \avrhoz-\eta_c\big)$ which in turn implies the Sobolev regularity of $\big(\rho(t,\cdot),\bu(t,\cdot)\big)$ by standard energy estimates. 

\medskip\noindent
{\bf Singular kernels}. We conclude by mentioning  existence results for the important class of singular kernels $\varphi_\beta(r):=r^{-\beta}$ with $\beta<d+2$: in this case, the  communication framework emphasizes  short-range  interactions  over long-range interactions, yet their global support still reflects global communication.
The regularity for 1D weakly singular kernels,  $\varphi_\beta$  with $\beta<1$ was  studied in \cite{Pes2015}, and for 1D strongly singular kernels with $1\leq \beta <3$ in \cite{ST2017a,ST2017b,ST2018a,DKRT2018}. In the latter case, alignment is structured as   fractional diffusion which was shown, at least in the one-dimensional case, to enforce \emph{unconditional} flocking behavior, independent of any initial threshold.  A typical result asserts that 
  \eqref{eq:noP} with strongly singular kernel, $\varphi_{1+\alpha}$ with $0<\alpha<2$ on $\T$, any non-vacuous initial data evolves into  a unique global solution,  $(\rho,u)\in L^\infty([0,\infty);  H^{s +\alpha} \times H^{s+1}), \ s\geq 3$, which converges  to a flocking traveling wave,
\[
	\|u(t,\cdot)-\overline{u}_0\|_{H^s} +\|\rho(t,\cdot) - \rho_\infty(\cdot-t\overline{u}_0) \|_{H^{s-1}} \lesssim  e^{-\eta t}, \quad t>0.
\]
Existence of strong solutions for  multiD problems with strongly singular kernels $\varphi(r)=r^{-(d+\alpha)}, \ 1<\alpha<2$  in $d\geq 3$-dimensions is open, except for  `small data' results  for   small initial data result  \cite{Sh2019} for
H\"older spaces, $|\bu_0-\bu_\infty|_\infty \lesssim (1+\|\rho_0\|_{W^{3,\infty}}+\|\bu_0\|_{W^{3,\infty}})^{-d}$ with $2/3<\alpha<3/2$, and in \cite{DMPW2019} for small Besov data 
\[
\|\bu_0\|_{B_{d,1}^{2-\alpha}} + \|\rho_0-1\|_{B_{d,1}^1} \leq \epsilon, \qquad  \alpha \in (1,2).
\]
There are fewer results on the existence of strong solutions  with  \emph{short range} interactions, in particular,  compactly supported $\varphi$'s and even the strongly singular kernels, $\varphi_\beta$ with $d< \beta< d+2$, exhibit   thinner tails than those sought for long-range communication \eqref{eq:long_range}. We mention the 1D dynamics driven by short-range topological-based kernels \cite{ST2020b}, 
\[
\begin{split}
 \phi(\bx,\bxp)&= \frac{{\mathds 1}_{R_0}(|\bx-\bxp|)}{|\bx-\bxp|^{\beta-\gamma}}\times\frac{1}{\mu^\gamma_\rho(\bx,\bxp)}, \qquad   0<\gamma< \beta,
 \end{split}
 \]
where $\displaystyle \mu_\rho(\bx,\bxp):= \left(\int_{{\mathcal C}(\bx,\bxp)}\rho(t,\bz)\dz\right)^{1/d}$ is the re-scaled mass in a communication region
${\mathcal C}(\bx,\bxp)$ enclosed between $\bx$ and $\bxp$.


\end{document}